\documentclass[12pt]{article}
\usepackage{amssymb,amsmath,amsthm,amscd,graphicx}
\newtheorem{remark}{Remark}[section]
\newtheorem{theorem}{Theorem}[section]
\newtheorem{proposition}{Propostion}[section]

\numberwithin{equation}{section}

\newcommand   {\F}  {\mathcal{F}}
\renewcommand {\i}  {\infty}

\renewcommand {\o}  {\omega}
\newcommand   {\B}  {\mathcal{B}}
\newcommand   {\D}  {\mathcal{D}}
\renewcommand {\O}  {\Omega}

\newcommand   {\pt} {\partial}
\newcommand   {\R}  {\mathbf{R}}

\begin{document}

\title{\bf Anticipating Stochastic $2D$ Navier-Stokes Equations}
\author{\ Salah Mohammed$^{1}$,
Tusheng Zhang$^{2}$}
\footnotetext[1]{\ Department of Mathematics,
Southern Illinois University, Carbondale, Illinois 62901, USA. \newline \indent Email: salah@sfde.math.siu.edu.
The research of this author is supported
by NSF award DMS-0705970.}
\footnotetext[2]{\ Department of Mathematics,
University of Manchester, Oxford Road, Manchester M13 9PL, England, \newline \indent
U.K. Email: tzhang@maths.man.ac.uk}

\maketitle
\begin{abstract}
In this article, we consider  the two-dimensional stochastic Navier-Stokes equation (SNSE) on a smooth bounded domain, driven by affine-linear multiplicative white noise and with random initial conditions and Dirichlet boundary conditions. The random initial condition is allowed to anticipate the forcing noise.  Our main objective is to prove the existence of a solution to the SNSE under sufficient Malliavin regularity of the initial condition. To this end we employ anticipating calculus techniques.
\end{abstract}

\allowdisplaybreaks
\section {Introduction. The main result.}

Two-dimensional stochastic Navier-Stokes equations (SNSE's) are often used to model the time
evolution of the velocity field for an incompressible fluid in a smooth bounded planar domain.
Existing models of fluid dynamics employ SNSE's with deterministic initial and boundary conditions.

Our main objective in this article is to prove existence of a variational  solution to the SNSE with 
random initial conditions that may possibly anticipate the driving noise.

\medskip
\medskip

The impetus for considering randomness in the initial condition for the SNSE is two-fold:
\begin{itemize}
\item{} Random measurement errors exist in physical models of hydrodynamic fluid movement.
\item{} Near stationary solutions, multiplicative ergodic theory techniques establish the existence of local random invariant manifolds that necessarily anticipate the driving noise of the SNSE ([MZ]). Thus a dynamic chcracterization of the semiflow along the invariant manifolds will require anlaysis of the SNSE with anticipating initial conditions. In particular, our main result in this article implies that each stationary point generates a stationary solution of the SNSE.
\end{itemize}


For simplicity of exposition we will only consider linear white noise with no additive colored noise. The case of affine noise (linear + additive) can be addressed via similar techniques and is left to the reader.

Consider the following two-dimensional stochastic Navier-Stokes
equation (SNSE) with Dirichlet boundary conditions and a random initial condition:
\begin{equation}
\left.
\begin{aligned}
du-\nu\triangle u\ dt+(u\cdot\nabla) u\ dt+\nabla
p\ dt&=\sum_{k=1}^{\infty} \sigma_k u(t)\,\circ dW_k(t) ,\\
(div\ u)(t,x)&=0,\ \quad  x\in D,\, t>0,\\
u(t,x)&=0, \quad x\in \partial D,\, t>0,\\
u(0,x)&=Y(x),\quad x\in D.
\end{aligned} \right \} 
\end{equation}
In the above SNSE, $D$ is a bounded domain in $\R^2$ with smooth boundary
$\partial D$, $u(t,x)\in\R^2 $ denotes the velocity field at
time $t$ and position $x \in D$, $p(t,x)$ denotes the pressure field, and
$\nu >0$ the viscosity coefficient.
Moreover, the random forcing field is provided by a family of independent one-dimensional standard Brownian motions $W_k, k \geq 1$, defined on a complete filtered Wiener space
$(\Omega, \F, (\F_t)_{t \geq 0}, P)$. We assume that the noise parameters $\sigma_k, \, k \geq 1,$ are such that $\displaystyle \sum_{k=1}^{\infty} \sigma_k^2 < \i$. The initial condition $Y$ is an $\F \otimes \B(D)$-measurable random field on $D$ where $\B(D)$ is the Borel $\sigma$-algebra of $D$.

\medskip

\medskip
 Under deterministic initial conditions, there is a large
amount of literature on the stochastic Navier-Stokes equation and its abstract setting. We will only refer to some of it.
A good reference for stochastic Navier-Stokes equations driven by
additive noise is the book \cite{D-Z.1} and the references therein.
The existence and uniqueness of solutions of stochastic $2D$
Navier-Stokes equations with multiplicative noise is established in
\cite{Fl.1} and \cite{S-S}. Ergodic properties and invariant measures
of stochastic $2D$ Navier-Stokes equations are studied in
\cite{Fl.1}
and \cite{H-M}. Large deviations under small noise and for occupation measures of stochastic $2D$ Navier-Stokes equations are studied in
\cite{S-S} and \cite{Gourcy}.
The existence of a $C^{1,1}$ cocycle and a mulitplicative ergodic theory for the SNSE (2.1) is established in \cite{M-Z}.

\medskip

In order to state our main result in this article, we consider the Hilbert space
$$
V:=\{v\in H^1_0(D,\R^2): \nabla\cdot v=0\  a.e. \,\, \mbox{in } D\},
$$
with the norm
$$
||v||_V:=\big(\int_D |\nabla v|^2\,dx\big)^{\frac{1}{2}}
$$
and associated inner product
$$\ll v_1,v_2 \gg_{V} := \int_D \nabla v_1\cdot \nabla v_2\,dx, \quad v_1, v_2 \in V.$$
 Denote by $H$ the closure of $V$ in the $L^2$-norm
$$
|v|_H:=\big(\int_D |v|^2\,dx\big)^{\frac{1}{2}}.
$$
The inner product on $H$ will be denoted by $<\cdot,\cdot>$.

\medskip

Our main result is the following existence theorem for  solutions of the SNSE (1.1):

\begin{theorem}
In the SNSE (1.1), assume that the initial random field $Y$ belongs to the Malliavin Sobolev space $\D^{1,4} (H)$ of all $\F$-measurable and Malliavin differentiable random variables $\O \to H$ with Malliavin derivatives having fourth-order moments. Then the SNSE (1.1) has a weak global solution with initial condition $Y$.
\end{theorem}

\section {Abstract formulation.}

\medskip
\indent To establish an abstract framework for the dynamics of the stochastic Navier-Stokes equation (1.1), we denote by $P_H$ the Helmholtz-Hodge projection of the Hilbert space
$L^2 (D,\R^2)$ onto the energy space $H$. Consider the (Stokes) operator $A$ in $H$ defined
by the formula
$$Au:=-\nu P_H\triangle u, \quad   u \in H^2(D,\R^2)\cap V,$$
 and the bilinear operator $B$ given by
$$B(u,v):=P_H\big((u\cdot\nabla) v\big),$$
whenever $u,v$ are such that $(u\cdot\nabla v)$ belongs to the space $L^2 (D, \R^2)$. We will often employ
the short notation $B(u):=B(u,u)$.

\indent By applying the operator $P_H$ to each term of the SNSE (1.1),
we can rewrite the latter equation in the following abstract form:
\begin{equation}
du(t)+Au(t)\, dt+B(u(t))\,dt= \sum_{k=1}^{\infty} \sigma_k u(t)\, \circ dW_k(t), \quad t >0,
\end{equation}
in $V'$ with the initial condition
\begin{equation}
u(0)=u_0 \in H.
\end{equation}
Here $V'$ stands for the dual of $V$.


Our approach is to identify the Hilbert space $H$ in Section 1 with its dual $H'$ and
consider the stochastic Navier-Stokes equation (1.1) in the framework of the Gelfand triple:
$$V\subset H\cong H'\subset V'.$$
%
Thus, we may consider the Stokes operator $A$ as a bounded linear map from $V$ into $V'$. Moreover, we also denote by $<\cdot,\cdot>: V \times V' \to \R$,
the canonical bilinear pairing between $V$ and $V'$. Hence, using integration by parts,
we have
\begin{equation}
<Au,w>=\nu\sum_{i,j=1}^2\int_D
\partial_iu_j\partial_iw_j dx=\nu \ll u,w \gg
\end{equation}
for $u=(u_1,u_2)\in V$, $w=(w_1,w_2)\in V$.

Define the real-valued trilinear form $b$ on $H\times H\times H$ by setting
\begin{equation}
b(u,v,w):=\sum_{i,j}^2\int_D u_i\partial_iv_jw_j dx,
\end{equation}
whenever the integral in (2.4) exists. In particular, if
$u,v,w \in V$, then
$$
b(u,v,w)=<B(u,v),w>=<(u\cdot\nabla)v,w>=\sum_{i,j}^2\int_D u_i\partial_iv_jw_j\,dx.
$$
Using integration by parts, it is easy to see that
\begin{equation}
b(u,v,w)=-b(u,w,v),
\end{equation}
for all $u,v,w \in V$.
Thus,
\begin{equation}\label{antisymmetric}
b(u,v,v)=0
\end{equation}
for all  $u,v \in V$.

Throughout the paper, we will denote various
generic positive constants by the same letter c, although the constants may differ from line to line. We now list some well-known estimates for $b$ which will be used frequently in the sequel (see \cite{Te}, \cite{Ro} for example):
%
\begin{align}
|b(u,v,w)|&\leq c \|u\|_V\cdot\|v\|_V\cdot\|w\|_V, \quad u,v,w \in V,\\
|b(u,v,w)|&\leq c |u|_H\cdot\|v\|_V\cdot|Aw|_H, \quad u \in H, v \in V, w \in D(A),\\
|b(u,v,w)|&\leq c \|u\|_V\cdot|v|_H\cdot|Aw|_H, \quad u \in V, v \in H, w \in D(A),\\
|b(u,v,w)|&\leq 2
\|u\|^{\frac{1}{2}}_V \cdot|u|^{\frac{1}{2}}_H\cdot\|w\|^{\frac{1}{2}}_V \cdot |w|^{\frac{1}{2}}_H\cdot\|v\|_V, \quad u,v,w \in V.
\end{align}
Moreover, combining (2.3) and (2.8), we obtain
\begin{equation}\label{2.9}
|B(u,w)|_{V'}=\sup_{\|v\|_V\leq 1}|b(u,w,v)|=\sup_{\|v\|_V \leq
1}|b(u,v,w)|\leq
2\|u\|^{\frac{1}{2}}_V\cdot|u|^{\frac{1}{2}}_H \cdot \|w\|^{\frac{1}{2}}_V\cdot|w|^{\frac{1}{2}}_H
\end{equation}
for all $u,w \in V$.

\section{Malliavin differentiability of the SNSE}

\setcounter{equation}{0}

In this section, we will show that the solutions of the SNSE (with non-random initial conditions) are Malliavin differentiable. Our approach is to use a variational technique which transforms the
SNSE (2.1) into a {\it random} Navier-Stokes equation that we can then analyze using a combination of Galerkin approximations and a priori estimates (cf. [Te], [Ro]).

 Consider the SNSE
\begin{eqnarray}\label{3.1}\left\{
\begin{array}{l}
du(t,f)+Au(t,f)dt+B(u(t,f))dt=
\displaystyle \sum_{k=1}^{\infty}\sigma_k u(t,f)\,\circ dW_{k}(t), \, t > 0,\\
u(0,f)=f \in H,
\end{array}\right.
\end{eqnarray}
with a deterministic initial condition $f \in H$.
It is known that for each $f \in H$, the SNSE (\ref{3.1}) admits  a unique  strong (in probabilistic sense)
solution $u(\cdot,f) \in L^2(\Omega;C([0,T];H))\cap L^2(\Omega\times
[0,T];V)$ ([B-C-F]). Writing (3.1) in integral form, we have
\begin{equation}\label{eq:3.2}
u(t,f)=f-\int_0^t Au(s,f)\,ds-\int_0^t B(u(s,f))\,ds+ \sum_{k=1}^{\infty}\int_0^t\sigma_k u(s,f))\,\circ dW_k(s),
\end{equation}
for all $t \in [0,T]$.

Let $Q: [0,\i) \times \O \to \R$ be the solution of the
one-dimensional linear sode
\begin{equation}    \label{eq:1.2}
\left. \begin{aligned}
  dQ(t) &= \sum_{k=1}^{\i} \sigma_k Q(t)\,\circ dW_k(t), \quad t \ge 0, \\
  Q(0) &= 1.
\end{aligned}  \right\}
\end{equation}
By It\^o's formula, we have
 \begin{equation}    \label{eq:1.2}
\begin{aligned}
  Q(t) = \exp\bigg \{\sum_{k=1}^{\i} \sigma_k W_k(t)\bigg \},  \quad t \ge 0.
\end{aligned}
\end{equation}
This implies that
$$E \|Q\|_{\i} < \i $$
where
$$\|Q\|_{\i} \equiv \|Q(\cdot,\o)\|_{\i}:= \sup_{0 \leq t \leq T} Q(t,\o), \quad \o \in \O,$$
for any finite positive $T$.
Define
\begin{equation}    \label{eq:1.3}
  v(t,f) := u(t,f)Q^{-1}(t), \quad t \geq 0.
\end{equation}
Applying It\^o{}'s formula to the relation $u(t,f) = v(t,f)Q(t), \,\, t \geq 0,$
%
%
and using (3.3), it is easy to see that $v(t) \equiv v(t,f)$ satisfies the random NSE
\begin{equation}
\left. \begin{aligned} \label{eq:3.6}
  dv(t) &= -Av(t)\,dt - Q(t) B\big(v(t)\big)\,dt, \, \quad t \geq 0,  \\
  v(0) &= f \in H\,.
\end{aligned} \right\}
\end{equation}
%
%
\vskip 0.3cm

\medskip
We recall the following two results from \cite{M-Z}.

\medskip

\begin{proposition} \label{prp:3.1}
For $f \in H$ and $\o \in \O$, let $v(\cdot,f,\o) \in C\big([0,T],H)\cap L^2\big ([0,T],
V)$ be a  solution of (3.6)  on $[0,T]$
for some $T > 0$. Then for each $\o \in \O$ and any $f \in H$, the following estimates hold
\begin{equation}    \label{3.7}
 \sup_{0 \leq t \leq T} |v(t,f,\o)|_{H} \le |f|_H
\end{equation}
%
and
\begin{equation}    \label{3.8}
  \int^T_0 \|  v(t,f,\o)\|^2_{V}dt
    \le \frac{1}{2\nu} |f|^2_{H}.
\end{equation}
Moreover, for each $\o \in \O$, the map $H \ni f \mapsto v(\cdot, f, \o) \in C\big([0,T],H)\cap
L^2\big ([0,T], V)$ is Lipschitz on bounded sets in $H$.
%
\end{proposition}

\medskip

\begin{proposition}  \label{th:3.1}
The solution map
$$H\ni f \rightarrow v(t,f,\o)\in H$$
of the random NSE (3.6) is $C^{1, 1}$
for each $\o \in \Omega$ and $t\geq 0$, and has Lipschitz
Fr\'e{}chet derivatives on bounded sets in $H$. Furthermore, the
Fr\'echet derivative $[0, \i) \ni t \to Dv(t,f,\o) \in L(H)$ is continuous in  $t$ and the following estimate holds:
\begin{equation}\label{3.24}
\sup_{0\le t\le T}\|D v(t,f)\|_{L(H)}\le
\exp \bigg (\frac{1}{2}\tilde c\|Q\|_{\infty}^2\frac{1}{2\nu}|f|_{H}^2 \bigg),
\end{equation}
where $L(H)$ denotes the space of bounded linear operators from $H$ into $H$.
\end{proposition}

The following result is an apriori energy estimate in the space $V$ for the random NSE (3.1).

\begin{proposition}
Let $v(t,f)$ be the solution to equation (\ref{eq:3.6}). For any $f\in V$, we have
\begin{eqnarray}\label{e0.1}
\sup_{0\leq t\leq T}||v(t,f)||_V^2+\nu\int_0^T|Av(s,f)|_H^2ds\nonumber\\
\leq ||f||_V^2 \exp\bigg (c |f|_H^4\sup_{0\leq s\leq T}Q^4(s)\bigg ).
\end{eqnarray}
\end{proposition}
\begin{proof}
By the chain rule, it follows that
\begin{eqnarray}\label{e0.2}
||v(t,f)||_V^2&=&||f||_V^2-2\nu \int_0^t|Av(s,f)|_H^2 \,ds\nonumber\\
&& \quad -2\int_0^t Q(s)<B(v(s,f)), Av(s,f)>\,ds.
\end{eqnarray}
Recall that (see Lemma 3.8 in \cite{Te1}) for $v\in H^2\cap V$,
\begin{equation}\label{e0.3}
|B(v)|_H\leq |v|_H^{\frac{1}{2}} ||v||_V |Av|_H^{\frac{1}{2}}.
\end{equation}
Thus
\begin{eqnarray}\label{e0.4}
2|Q(s)<B(v(s,f)),Av(s,f)>|
&\leq& 2Q(s)|v(s,f)|_H^{\frac{1}{2}} ||v(s,f)||_V |Av(s,f)|_H^{\frac{3}{2}}\nonumber\\
&\leq & \nu |Av(s,f)|_H^{2}+c Q^4(s)|v(s,f)|_H^{2} ||v(s,f)||_V^4\nonumber\\
&\leq & \nu |Av(s,f)|_H^{2}+c Q^4(s)|f|_H^{2} ||v(s,f)||_V^4,
\end{eqnarray}
where we have used (\ref{3.7}). By relations
(\ref{e0.2}) and (\ref{e0.4}), we obtain
\begin{eqnarray}\label{e0.5}
||v(t,f)||_V^2&\leq &||f||_V^2-\nu \int_0^t|Av(s,f)|_H^2ds\nonumber\\
&& \quad + c |f|_H^{2}\int_0^tQ^4(s)||v(s,f)||_V^2||v(s,f)||_V^2 ds.
\end{eqnarray}

Applying Gronwall's inequality to (\ref{e0.5}) and using (\ref{3.8}), we get
\begin{eqnarray}\label{e0.6}
\sup_{0\leq t\leq T}||v(t,f)||_V^2+\nu \int_0^T|Av(s,f)|_H^2\,ds
&\leq &||f||_V^2 \exp(c |f|_H^2\sup_{0\leq s\leq T}Q^4(s)\int_0^T||v(s,f)||_V^2ds)\nonumber\\
&\leq &||f||_V^2 \exp(c |f|_H^4\sup_{0\leq s\leq T}Q^4(s)).
\end{eqnarray}

\end{proof}

\begin{remark}
Set $\sigma:=\sqrt{\sum_{k=1}^{\infty}\sigma_k^2}$. Define
$$W(t):=\frac{1}{\sigma}\sum_{k=1}^{\infty}\sigma_kW_k(t), \quad t \geq 0.$$
Then $W(t), t\geq 0,$ is a new one-dimensional standard Brownian motion and
$$ \sum_{k=1}^{\infty}\sigma_ku(t,f)\circ dW_k(t)=\sigma u(t,f)\circ dW(t).$$
Thus, from now on and without loss of generality, we will assume that the SNSE (2.1) is driven by
one single Brownian motion $W$ and with $\sigma=1$. Hence $Q(t)=\exp(W(t)), \, t \geq 0$.
\end{remark}

\medskip

We next develop Malliavin derivatives for solutions of the random NSE (3.6). For the rest of the article we will denote Malliavin derivatives by $\D$.

\begin{proposition}
For each $f\in V$ and $t \geq 0$, the solution map
$$\Omega \ni \omega \rightarrow v(t,f,\o)\in H$$
of (3.6) is Malliavin differentiable. Its Malliavin  derivative $\D_uv(t,f)$, solves the following random evolution equation:
\begin{eqnarray}\label{e0.6}
\D_uv(t,f)\nonumber
&=& -\int_0^t A\D_uv(s,f)\, ds-\int_0^t Q(s)(\D_uv(s,f)\cdot \nabla )v(s,f)\,ds\nonumber\\
&&-\int_0^t Q(s)(v(s,f)\cdot \nabla )\D_uv(s,f) ds-\int_0^t \D_uQ(s)(v(s,f)\cdot \nabla )v(s,f)\, ds \nonumber\\
\end{eqnarray}
for all $t \in [0,T]$.
\end{proposition}

\begin{proof}
We will prove $v(t,f)\in \D^{1,2}_{loc}(H)$. By the uniqueness of the solution of the random NSE (\ref{eq:3.6}), we have
$v(t,f)=v^N(t,f)$ on $\Omega_N=\{ \sup_{0\leq s\leq T}|W(s)|\leq N\}$, where $v^N(t,f)$ is the solution of an equation similar to (\ref{eq:3.6}) replacing $Q(s)$ there by $Q_N(s)=\exp(W(s)\wedge N)$. Thus it is sufficient to prove $v^N(t,f)\in \D^{1,2}(H)$ for every fixed $N$.  For this reason, we assume implicitly in the proof that $Q=Q_N$.
To continue, we appeal to the Galerkin approximations. Let $\{ e_i\}_{i=1}^\i$ be a complete orthonormal
basis of $H$ that consists of eigenvectors of the operator $-A$ under Dirichlet boundary conditions with corresponding eigenvalues $\{\mu_i\}_{i=1}^\i$\,; that is $A(e_i) = -\mu_ie_i$, $e_i|_{\pt D} = 0, \,\, i \ge 1$. Let
$H_n$ denote the $n$-dimensional subspace of $H$ spanned by $\{e_1, e_2, ...,e_n\}$. Define
$f_n \in H_n$ by
\[
  f_n := \sum^n_{j=1} \langle f,e_j \rangle e_j\,.
\]
Clearly, the sequence $\{f_n\}_{n=1}^\i$ converges to $f$ in $H$.
Now for every integer $n \geq 1$,  let $v_n$ be unique solution of  the
random NSE
\begin{equation}\label{e0.7}
\left. \begin{aligned}
  dv_n(t,f_n) &= -Av_n(t,f_n)\,dt - Q(t)B\big(v_n(t,f_n)\big)\,dt, \quad t > 0, \\
  v_n(0,f_n) &= f_n, \\
\ v_n(t,f_n) \mid_{\pt D} &= 0, \quad t > 0,
\end{aligned} \right\}
\end{equation}
such that
\[
  v_n(t, f_n) := \sum^n_{j=1} g_{jn}(t) e_j, \quad t \geq 0,
\]
for appropriate choice of the real-valued random processes $g_{jn}$. It was shown in \cite{M-Z} that $v_n$ converges to $v$ and
\begin{equation}\label{e0.8}
\lim_{n\rightarrow \infty}E[\int_0^T|v_n(s,f_n)-v(s,f)|_H^2ds]=0.
\end{equation}
As $v_n(t,f_n)$ is a  solution of the finite dimensional random ordinary differential equation (\ref{e0.7}), it is known (see e.g. \cite{N}) that $v_n(t,f_n)$ is Malliavin differentiable and the corresponding Malliavin derivative $\D_uv_n(t,f_n)$ satisfies the following random ODE:
\begin{eqnarray}\label{e0.9}
\D_uv_n(t,f_n)
&=& -\int_0^t A\D_uv_n(s,f_n) ds-\int_0^t Q(s)(\D_uv_n(s,f_n)\cdot \nabla )v_n(s,f_n)ds\nonumber\\
&&-\int_0^t Q(s)(v_n(s,f_n)\cdot \nabla )\D_uv_n(s,f_n) ds-\int_0^t \D_uQ(s)(v_n(s,f_n)\cdot \nabla )v_n(s,f_n) ds \nonumber\\
&& \quad\quad \quad  \quad
\end{eqnarray}
for all $t \in [0,T]$.
Let $Y_u(t,f)$ be the solution of the following random evolution equation:
\begin{eqnarray}\label{e0.10}
Y_u(t,f)
&=& -\int_0^t AY_u(s,f) ds-\int_0^t Q(s)(Y_u(s,f)\cdot \nabla )v(s,f)ds\nonumber\\
&&-\int_0^t Q(s)(v(s,f)\cdot \nabla )Y_u(s,f) ds-\int_0^t D_uQ(s)(v(s,f)\cdot \nabla )v(s,f) ds 
\end{eqnarray}
for all $t \in [0,T]$.
The existence of the solution of the above equation can be obtained by a similar method to the one used for (\ref{eq:3.6}) (see \cite{M-Z}). Since the Malliavin derivative operator $\D$ is closed, to prove the theorem it suffices to show that
\begin{equation}\label{e0.11}
\lim_{n\rightarrow \infty}\sup_{0\leq u\leq t}E[|\D_uv_n(t,f_n)-Y_u(t,f) |_H^2]=0.
\end{equation}
Now,
\begin{eqnarray}\label{e0.12}
&&|\D_uv_n(t,f_n)-Y_u(t,f) |_H^2\nonumber\\
&=& -2\nu \int_0^t ||\D_uv_n(s,f_n)-Y_u(s,f)||_V^2 ds\nonumber\\
&&-2\int_0^t \D_uQ(s)b(v_n(s,f_n), v_n(s,f_n)-v(s,f), \D_uv_n(s,f_n)-Y_u(s,f))ds\nonumber\\
&&-2\int_0^t \D_uQ(s)b(v_n(s,f_n)-v(s,f), v(s,f), \D_uv_n(s,f_n)-Y_u(s,f))ds\nonumber\\
&&-2\int_0^t Q(s)b(\D_uv_n(s,f_n), v_n(s,f_n)-v(s,f), \D_uv_n(s,f_n)-Y_u(s,f))ds\nonumber\\
&&-2\int_0^t Q(s)b(\D_uv_n(s,f_n)-Y_u(s,f), v(s,f), \D_uv_n(s,f_n)-Y_u(s,f))ds\nonumber\\
&&-2\int_0^t Q(s)b(v_n(s,f_n)-v(s,f), \D_uv_n(s,f_n), \D_uv_n(s,f_n)-Y_u(s,f))ds\nonumber\\
&&-2\int_0^t Q(s)b(v(s,f), \D_uv_n(s,f_n)-Y_u(s,f), \D_uv_n(s,f_n)-Y_u(s,f))ds \nonumber\\
&:=&I^n_1+I^n_2+I^n_3+I^n_4+I^n_5+I^n_6+I^n_7 \quad t \in [0,T].
\end{eqnarray}
Set
$$C^n_1(\omega)=\sup_{0\leq s\leq T}|v_n(s,f_n)|_H,  \quad C^n_2(\omega)=\sup_{0\leq s\leq T}||v_n(s,f_n)||_V$$
$$C_1(\omega)=\sup_{0\leq s\leq T}|v(s,f)|_H, \quad C_2(\omega)=\sup_{0\leq s\leq T}||v(s,f)||_V$$
$$M^n_1(u, \omega)=\sup_{0\leq s\leq T}|\D_uv_n(s,f_n)|_H,  M^n_2(u, \omega)=\sup_{0\leq s\leq T}||\D_uv_n(s,f_n)||_V$$
$$M_1(u, \omega)=\sup_{0\leq s\leq T}|Y_u(s,f)|_H, M_2(u, \omega)=\sup_{0\leq s\leq T}||Y_u(s,f)||_V$$
Now we estimate each of the terms on the right of (\ref{e0.12}). We start with
\begin{eqnarray}\label{e0.13}
I^n_2
&\leq & c \sup_{0\leq s\leq T}|\D_uQ(s)| \int_0^t |v_n(s,f_n)|_H^{\frac{1}{2}}||v_n(s,f_n)||_V^{\frac{1}{2}} |v_n(s,f_n)-v(s,f)|_H^{\frac{1}{2}} \nonumber\\
&&\quad\quad\quad \times ||v_n(s,f_n)-v(s,f)||_V^{\frac{1}{2}}||\D_uv_n(s,f_n)-Y_u(s,f))||_Vds\nonumber\\
&\leq & c \sup_{0\leq s\leq T}|\D_uQ(s)| (C^n_1(\omega)+C_1(\omega))(C^n_2(\omega)+C_2(\omega))(\int_0^T |v_n(s,f_n)-v(s,f)|_H^2ds)^{\frac{1}{4}}\nonumber\\
&&\quad\quad\quad \times (\int_0^T ||\D_uv_n(s,f_n)-Y_u(s,f))||_V^{\frac{4}{3}}ds)^{\frac{3}{4}}
\end{eqnarray}
and
\begin{eqnarray}\label{e0.14}
I^n_3&\leq & c \sup_{0\leq s\leq T}|\D_uQ(s)| (C^n_1(\omega)+C_1(\omega))(C^n_2(\omega)+C_2(\omega))(\int_0^T |v_n(s,f_n)-v(s,f)|_H^2ds)^{\frac{1}{4}}\nonumber\\
&&\quad\quad\quad \times (\int_0^T ||\D_uv_n(s,f_n)-Y_u(s,f))||_V^{\frac{4}{3}}ds)^{\frac{3}{4}}
\end{eqnarray}
For $I^n_4$, we have
\begin{eqnarray}\label{e0.15}
I^n_4
&\leq & c \sup_{0\leq s\leq T}|Q(s)| \int_0^t |\D_uv_n(s,f_n)|_H^{\frac{1}{2}}||\D_uv_n(s,f_n)||_V^{\frac{1}{2}} |v_n(s,f_n)-v(s,f)|_H^{\frac{1}{2}}\nonumber\\
&&\quad \quad \quad  ||v_n(s,f_n)-v(s,f)||_V^{\frac{1}{2}}||\D_uv_n(s,f_n)-Y_u(s,f))||_Vds\nonumber\\
&\leq & c ||Q||_{\infty} [M^n_1(u, \omega)M^n_2(u, \omega)(C^n_2(\omega)+C_2(\omega))]^{\frac{1}{2}}(\int_0^T |v_n(s,f_n)-v(s,f)|_H^2ds)^{\frac{1}{4}}\nonumber\\
&&\quad\quad\quad \times (\int_0^T ||\D_uv_n(s,f_n)-Y_u(s,f))||_V^{\frac{4}{3}}ds)^{\frac{3}{4}}
\end{eqnarray}
The term $I^n_5$ can be bounded as follows:
\begin{eqnarray}\label{e0.16}
I^n_5
&\leq & c \sup_{0\leq s\leq T}|Q(s)| \int_0^t ||v(s,f)||_V ||\D_uv_n(s,f_n)-Y_u(s,f))||_V||\D_uv_n(s,f_n)-Y_u(s,f))||_H ds\nonumber\\
&\leq & +\frac{\nu}{2}\int_0^t||\D_uv_n(s,f_n)-Y_u(s,f))||_V^2ds\nonumber\\
  && +c_{\nu} ||Q||_{\infty}^2 \int_0^t ||v(s,f)||_V^2|\D_uv_n(s,f_n)-Y_u(s,f))|_H^2ds
\end{eqnarray}
Now,
\begin{eqnarray}\label{e0.17}
I^n_6
&\leq & c \sup_{0\leq s\leq T}|Q(s)| \int_0^t ||\D_uv_n(s,f_n)||_V |v_n(s,f_n)-v(s,f)|_H^{\frac{1}{2}} ||v_n(s,f_n)-v(s,f)||_V^{\frac{1}{2}}\nonumber\\
&&\quad\quad\quad ||\D_uv_n(s,f_n)-Y_u(s,f))||_V^{\frac{1}{2}}|\D_uv_n(s,f_n)-Y_u(s,f))|_H^{\frac{1}{2}} ds\nonumber\\
&\leq & c ||Q||_{\infty} M^n_2(u, \omega)[(M^n_1(u, \omega)+M_1(u, \omega))(C^n_2(u, \omega)+C_2(\omega))]^{\frac{1}{2}}(\int_0^T |v_n(s,f_n)-v(s,f)|_H^2ds)^{\frac{1}{4}}\nonumber\\
&&\quad\quad\quad \times (\int_0^T ||\D_uv_n(s,f_n)-Y_u(s,f))||_V^{\frac{4}{3}}ds)^{\frac{3}{4}}
\end{eqnarray}
Finally, we note that $I^n_7=0$ because $b(u,v,v)=0$ for all $u,v \in V$.

Substituting (\ref{e0.13})--(\ref{e0.17}) into (\ref{e0.12}) and applying Gronwall's inequality  we obtain
\begin{eqnarray}\label{e0.19}
\sup_{u\leq t\leq T}|\D_uv_n(t,f_n)-Y_u(t,f) |_H^2&+\nu \int_0^T ||\D_uv_n(s,f_n)-Y_u(s,f)||_V^2 ds\nonumber\\
&\leq L_n(\omega) \exp\bigg ( c_{\nu} ||Q||_{\infty}^2 \int_0^T ||v(s,f)||_V^2ds\bigg )\nonumber\\
&\leq L_n(\omega) \exp\bigg ( c_{\nu} ||Q||_{\infty}^2 |f|_H^2\bigg )
\end{eqnarray}
where $L_n(\omega)$ is the sum of the right sides of (\ref{e0.13}), (\ref{e0.14}), (\ref{e0.15}), (\ref{e0.17}).
By the dominated convergence theorem, we deduce that
 \begin{eqnarray}\label{e0.20}
&&E[\sup_{u\leq t\leq T}|\D_uv_n(t,f_n)-Y_u(t,f) |_H^2]+\nu E[\int_0^T ||\D_uv_n(s,f_n)-Y_u(s,f)||_V^2 ds]\nonumber\\
&& \rightarrow 0, \quad \mbox{as} \quad n\rightarrow \infty,
\end{eqnarray}
where Proposition 3.3 has been used.
\end{proof}

\begin{proposition}
For any $f\in V$, the Malliavin  derivative $\D_uv(t,f)$ of $v(t,f)$ satisfies the following estimate:
\begin{eqnarray}\label{e0.32}
\sup_{u\leq t\leq T}|\D_uv(t,f)|_H^2]+\nu E[\int_0^T ||\D_uv(s,f)||_V^2 ds]
\leq C_{\nu} ||\D_uQ||_{\infty}^2 |f|_H^4 \exp(C||Q||_{\infty}^2 |f|_H^2).
\end{eqnarray}
\end{proposition}
\begin{proof}
By chain rule,
 \begin{eqnarray}\label{e0.33}
|\D_uv(t,f)|_H^2
&=& -2\nu \int_0^t ||\D_uv(s,f)||_V^2 ds\nonumber\\
&&-2\int_0^t Q(s)b(\D_uv(s,f), v(s,f), \D_uv(s,f))ds\nonumber\\
&&-2\int_0^t Q(s)b(v(s,f), \D_uv(s,f), \D_uv(s,f))ds\nonumber\\
&&-2\int_0^t \D_uQ(s)b(v(s,f), v(s,f), \D_uv(s,f))ds\nonumber\\
&:=&K_1+K_2+K_3+K_4 \quad t \in [0,T].
\end{eqnarray}
In view of (\ref{antisymmetric}), $K_3=0$. For the other terms $K_2, K_4$, the following estimates hold:
 \begin{eqnarray}\label{e0.34}
K_2
&\leq &c ||Q||_{\infty}\int_0^t ||\D_uv(s,f)||_V ||v(s,f)||_V |\D_uv(s,f)|_Hds\nonumber\\
&\leq &\frac{\nu}{2}\int_0^t ||\D_uv(s,f)||_V^2 ds+c_{\nu}||Q||_{\infty}^2\int_0^t  ||v(s,f)||_V^2 |\D_uv(s,f)|_H^2ds
\end{eqnarray}
and
\begin{eqnarray}\label{e0.35}
K_4
&\leq &c ||\D_uQ||_{\infty}\int_0^t ||\D_uv(s,f)||_V ||v(s,f)||_V |v(s,f)|_Hds\nonumber\\
&\leq &\frac{\nu}{2}\int_0^t ||\D_uv(s,f)||_V^2 ds+c_{\nu}||\D_uQ||_{\infty}^2\int_0^t  ||v(s,f)||_V^2 |v(s,f)|_H^2ds\nonumber\\
&\leq& \frac{\nu}{2}\int_0^t ||\D_uv(s,f)||_V^2 ds+c_{\nu}||\D_uQ||_{\infty}^2|f|_H^4,
\end{eqnarray}
where (\ref{3.7}), (\ref{3.8}) were used.
Now (\ref{e0.32}) follows from (\ref{e0.34}) and (\ref{e0.35}).

\end{proof}

\begin{theorem}
For each $f\in H$, the solution map
$$\Omega \ni \omega \rightarrow v(t,f,\o)\in H$$
of the random NSE (3.6) is Malliavin differentiable. Its Malliavin  derivative $\D_uv(t,f)$ solves the following random evolution equation:
\newpage
\begin{eqnarray}\label{e0.21}
\D_uv(t,f)\nonumber
&=& -\int_0^t A\D_uv(s,f) ds-\int_0^t Q(s)(\D_uv(s,f)\cdot \nabla )v(s,f)ds\nonumber\\
&&-\int_0^t Q(s)(v(s,f)\cdot \nabla )\D_uv(s,f) ds-\int_0^t \D_uQ(s)(v(s,f)\cdot \nabla )v(s,f) ds,\nonumber\\
\end{eqnarray}
for all $t \in [0,T]$.
\end{theorem}
\begin{proof}
Again, as in the proof of Proposition 3.4, we will implicitly assume $Q=Q_N$.
Take $f_n\in V, n\geq 1$ such that $f_n\rightarrow f$ in $H$ as $n\rightarrow \infty$. By Proposition 3.4 , we know that
  $v(t,f_n,\o)\in H$
is Malliavin differentiable. The Malliavin  derivative $\D_uv(t,f_n)$ solves the following random evolution equation:
\begin{eqnarray}\label{e0.21}
\D_uv(t,f_n)
&&=-\int_0^t A\D_uv(s,f_n) ds-\int_0^t Q(s)(\D_uv(s,f_n)\cdot \nabla )v(s,f_n)ds\nonumber\\
&&-\int_0^t Q(s)(v(s,f_n)\cdot \nabla )\D_uv(s,f_n) ds-\int_0^t \D_uQ(s)(v(s,f_n)\cdot \nabla )v(s,f_n) ds,\nonumber\\
\end{eqnarray}
for all $t \in [0,T]$.
On the other hand, it follows from Proposition 3.1 that
$$\lim_{n\rightarrow \infty}E[|v(t,f_n)-v(t,f)|_H^p]=0$$
$$\lim_{n\rightarrow \infty}E\bigg (\int_0^T||v(t,f_n)-v(t,f)||_V^2dt \bigg )^p=0$$
for any $p>0$. Thus it suffices to show that the Malliavin derivatives $\D_uv(t, f_n), n\geq 1,$ converge. Let $Z_u(t,f)$ be the solution of the following random evolution equation:
\begin{eqnarray}\label{e0.22}
Z_u(t,f)
&=& -\int_0^t AZ_u(s,f) ds-\int_0^t Q(s)(Z_u(s,f)\cdot \nabla )v(s,f)ds\nonumber\\
&&-\int_0^t Q(s)(v(s,f)\cdot \nabla )Z_u(s,f) ds-\int_0^t \D_uQ(s)(v(s,f)\cdot \nabla )v(s,f) ds, \nonumber\\
\end{eqnarray}
for all $t \in [0,T]$.
 Since  the Malliavin derivative operator $\D$ is closed, to prove the theorem it suffices to show that
\begin{equation}\label{e0.23}
\lim_{n\rightarrow \infty}\sup_{0\leq u\leq t}E[|\D_uv_n(t,f_n)-Z_u(t,f) |_H^2]=0.
\end{equation}

To prove (3.37), consider the following:

\begin{eqnarray}\label{e0.24}
|\D_uv(t,f_n)-Z_u(t,f) |_H^2
&=& -2\nu \int_0^t ||\D_uv(s,f_n)-Z_u(s,f)||_V^2 ds\nonumber\\
&&-2\int_0^t \D_uQ(s)b(v(s,f_n), v(s,f_n)-v(s,f), \D_uv(s,f_n)-Z_u(s,f))ds\nonumber\\
&&-2\int_0^t \D_uQ(s)b(v(s,f_n)-v(s,f), v(s,f), \D_uv(s,f_n)-Z_u(s,f))ds\nonumber\\
&&-2\int_0^t Q(s)b(\D_uv(s,f_n), v(s,f_n)-v(s,f), \D_uv(s,f_n)-Z_u(s,f))ds\nonumber\\
&&-2\int_0^t Q(s)b(\D_uv(s,f_n)-Z_u(s,f), v(s,f), \D_uv(s,f_n)-Z_u(s,f))ds\nonumber\\
&&-2\int_0^t Q(s)b(v(s,f_n)-v(s,f), \D_uv(s,f_n), \D_uv(s,f_n)-Z_u(s,f))ds\nonumber\\
&&-2\int_0^t Q(s)b(v(s,f), \D_uv(s,f_n)-Z_u(s,f), \D_uv(s,f_n)-Z_u(s,f))ds \nonumber\\
&:=&J^n_1+J^n_2+J^n_3+J^n_4+J^n_5+J^n_6+J^n_7
\end{eqnarray}
for all $t \in [0,T]$. Note first that $J^n_7=0$ because the trilinear form $b(\cdot,\cdot,\cdot)$ is anti-symmetric with respect to the last two arguments. To estimate $J_2^n$ and $J_3^n$, note that
\begin{eqnarray}\label{e0.25}
J^n_2
&\leq & c \sup_{0\leq s\leq T}|\D_uQ(s)| \int_0^t |v(s,f_n)|_H^{\frac{1}{2}}||v(s,f_n)||_V^{\frac{1}{2}} |v(s,f_n)-v(s,f)|_H^{\frac{1}{2}} \nonumber\\
&&\quad\quad\quad \times ||v(s,f_n)-v(s,f)||_V^{\frac{1}{2}}||\D_uv(s,f_n)-Z_u(s,f))||_Vds\nonumber\\
&\leq & c \sup_{0\leq s\leq T}|\D_uQ(s)| \sup_{0\leq s\leq T}(|v(s,f_n)-v(s,f)|_H^{\frac{1}{2}})|f_n|_H^{\frac{1}{2}}\nonumber\\
&&\times \int_0^t ||v(s,f_n)||_V^{\frac{1}{2}} ||v(s,f_n)-v(s,f)||_V^{\frac{1}{2}}||\D_uv(s,f_n)-Z_u(s,f))||_Vds\nonumber\\
&&\longrightarrow 0, \quad \text{as}\quad n\rightarrow \infty.
\end{eqnarray}
and
\begin{eqnarray}\label{e0.26}
J^n_3
&\leq & c \sup_{0\leq s\leq T}|\D_uQ(s)| \int_0^t |v(s,f)|_H^{\frac{1}{2}}||v(s,f)||_V^{\frac{1}{2}} |v(s,f_n)-v(s,f)|_H^{\frac{1}{2}} \nonumber\\
&&\quad\quad\quad \times ||v(s,f_n)-v(s,f)||_V^{\frac{1}{2}}||\D_uv(s,f_n)-Z_u(s,f))||_Vds\nonumber\\
&\leq & c \sup_{0\leq s\leq T}|\D_uQ(s)| \sup_{0\leq s\leq T}(|v(s,f_n)-v(s,f)|_H^{\frac{1}{2}})|f|_H^{\frac{1}{2}}\nonumber\\
&&\times \int_0^t ||v(s,f)||_V^{\frac{1}{2}} ||v(s,f_n)-v(s,f)||_V^{\frac{1}{2}}||\D_uv(s,f_n)-Z_u(s,f))||_Vds\nonumber\\
&&\longrightarrow 0, \quad \text{as}\quad n\rightarrow \infty.
\end{eqnarray}
For $J^n_4$, we have
\begin{eqnarray}\label{e0.28}
J^n_4
&\leq & c \sup_{0\leq s\leq T}|Q(s)| \int_0^t |\D_uv(s,f_n)|_H^{\frac{1}{2}}||\D_uv(s,f_n)||_V^{\frac{1}{2}} |v_n(s,f_n)-v(s,f)|_H^{\frac{1}{2}}\nonumber\\
&&\quad \quad \quad  ||v_n(s,f_n)-v(s,f)||_V^{\frac{1}{2}}||\D_uv_n(s,f_n)-Z_u(s,f))||_Vds\nonumber\\
&\leq & c ||Q||_{\infty}\sup_{0\leq s\leq T}(|v(s,f_n)-v(s,f)|_H^{\frac{1}{2}})\sup_{0\leq s\leq T}(|\D_uv(s,f_n)|_H^{\frac{1}{2}})\nonumber\\
&&\quad\quad\quad \times (\int_0^T||\D_uv(s,f_n)||_V^{\frac{1}{2}}||v_n(s,f_n)-v(s,f)||_V^{\frac{1}{2}}
||\D_uv_n(s,f_n)-Z_u(s,f))||_Vds\nonumber\\
&&\longrightarrow 0, \quad \text{as}\quad n\rightarrow \infty.
\end{eqnarray}
The term $J^n_5$ can be estimated as follows:
\begin{eqnarray}\label{e0.29}
J^n_5
&\leq & c \sup_{0\leq s\leq T}|Q(s)| \int_0^t ||v(s,f)||_V ||\D_uv(s,f_n)-Z_u(s,f))||_V |\D_uv(s,f_n)-Z_u(s,f))|_H ds\nonumber\\
&\leq & +\frac{\nu}{2}\int_0^t||\D_uv_n(s,f_n)-Y_u(s,f))||_V^2ds\nonumber\\
  && +c_{\nu} ||Q||_{\infty}^2 \int_0^t ||v(s,f)||_V^2|\D_uv_n(s,f_n)-Y_u(s,f))|_H^2\,ds.
\end{eqnarray}
Furthermore,
\begin{eqnarray}\label{e0.30}
J^n_6
&\leq & c \sup_{0\leq s\leq T}|Q(s)| \int_0^t ||\D_uv(s,f_n)||_V |v(s,f_n)-v(s,f)|_H^{\frac{1}{2}} ||v(s,f_n)-v(s,f)||_V^{\frac{1}{2}}\nonumber\\
&&\quad\quad\quad ||\D_uv_n(s,f_n)-Y_u(s,f))||_V|\D_uv(s,f_n)|_H^{\frac{1}{2}} ds\nonumber\\
&\leq & c ||Q||_{\infty}\sup_{0\leq s\leq T}(|v(s,f_n)-v(s,f)|_H^{\frac{1}{2}}\sup_{0\leq s\leq T}(|\D_uv(s,f_n)|_H^{\frac{1}{2}}\nonumber\\
&&\times \int_0^t ||\D_uv(s,f_n)||_V^{\frac{1}{2}} ||v(s,f_n)-v(s,f)||_V^{\frac{1}{2}}||\D_uv_n(s,f_n)-Y_u(s,f))||_V ds
\nonumber\\
&&\longrightarrow 0, \quad \text{as}\quad n\rightarrow \infty.
\end{eqnarray}

Substituting (\ref{e0.25})--(\ref{e0.30}) into (\ref{e0.24}) and applying Gronwall's inequality,  we obtain
\begin{eqnarray}\label{e0.31}
\sup_{u\leq t\leq T}|\D_uv(t,f_n)-Z_u(t,f) |_H^2&&+\nu \int_0^T ||\D_uv(s,f_n)-Z_u(s,f)||_V^2 ds\nonumber\\
&\leq &\tilde{L}_n(\omega) \exp\bigg ( c_{\nu} ||Q||_{\infty}^2 \int_0^T ||v(s,f)||_V^2ds\bigg )\nonumber\\
&\leq &\tilde{L}_n(\omega) \exp\bigg ( c_{\nu} ||Q||_{\infty}^2 |f|_H^2\bigg )
\end{eqnarray}
where $\tilde{L}_n(\omega)$ is the sum of the right-hand sides of (\ref{e0.25}), (\ref{e0.26}), (\ref{e0.28}), (\ref{e0.30}) and is such that $\tilde{L}_n\rightarrow 0$ as $n\rightarrow \infty$. Finally, (\ref{e0.23}) follows from
dominated convergence theorem. \end{proof}

\section{The Anticipating SNSE}

We are now ready to state our main result whereby we replace the deterministic initial function $f$ in the the SNSE (3.1) by an anticipating random field $Y \in \D^{1,4}(H)$:
\begin{theorem}
Suppose $Y\in \D^{1,4}(H)$. Then $u(t,Y), t\geq 0$, solves the following anticipating Stratonovich SNSE:
\begin{equation}\label{e1.1}
u(t,Y)=Y-\int_0^tAu(s,Y)ds -\int_0^tB(u(s,Y))ds +\int_0^tu(s,Y)\circ dW(s).
\end{equation}
\end{theorem}
\begin{proof}
 Note that $u(t,Y)$ takes values in $H$. But $Au(s,Y)$ and $B(u(s,Y))$ belong to $V^{\prime}$. Because of this special infinite-dimensional setting, the existing chain rules in the literature could not be applied. We will therefore give a direct proof.

Fix $t>0$ and let $\{ 0=t_0^n<t_1^n<...<t_{k_n}^n=t\}, n\geq 1$ be a sequence of partitions of the interval $[0,t]$ such that $\tau^n=max_i(t^n_{i+1}-t^n_i)\rightarrow 0$ as $n\rightarrow \infty$. Write
\begin{eqnarray}\label{e1.2}
u(t,Y)-Y&=&v(t,Y)Q(t)-Y \nonumber\\
&= &\sum_{i=0}^{k_n-1}(v(t_{i+1},Y)Q(t_{i+1})-v(t_i,Y)Q(t_i)) \nonumber\\
&= &\sum_{i=0}^{k_n-1}Q(t_{i+1})(v(t_{i+1},Y)-v(t_i,Y))+\sum_{i=0}^{k_n-1}v(t_i,Y)(Q(t_{i+1})-Q(t_{i}))\nonumber\\
&=&-\sum_{i=0}^{k_n-1}\int_{t_i}^{t_{i+1}}Q(t_{i+1})Av(s,Y)ds
-\sum_{i=0}^{k_n-1}\int_{t_i}^{t_{i+1}}Q(t_{i+1})Q(s)B(v(s,Y))ds\nonumber\\
&+&\sum_{i=0}^{k_n-1}v(t_i,Y)\int_{t_i}^{t_{i+1}}Q(s)dW(s)+\frac{1}{2}\sum_{i=0}^{k_n-1}v(t_i,Y)\int_{t_i}^{t_{i+1}}Q(s)\,ds\nonumber\\
&:=& T^n_1+T^n_2+T^n_3+T^n_4.
\end{eqnarray}
As $v(s,Y), Q(s)$ are continuous in $s$, clearly we have
\begin{eqnarray}\label{e1.3}
&&\lim_{n\rightarrow\infty}T^n_1=-\int_0^tAu(s,Y)ds, \\ &&\lim_{n\rightarrow\infty}T^n_2=-\int_0^tQ(s)^2B(v(s,Y))ds=-\int_0^tB(u(s,Y))ds
\end{eqnarray}
and
\begin{equation}\label{e1.4}
\lim_{n\rightarrow\infty}T^n_4=\frac{1}{2}\int_0^tQ(s)v(s,Y)ds=\frac{1}{2}\int_0^tu(s,Y)ds.
\end{equation}
By the property of the Skorohod integral ([N], p. 40), we have
 \begin{eqnarray}\label{e1.5}
T^n_3
&=&\sum_{i=0}^{k_n-1}\int_{t_i}^{t_{i+1}}v(t_i,Y)Q(s)dW(s)+\sum_{i=0}^{k_n-1}\int_{t_i}^{t_{i+1}}\D_s(v(t_i,Y))Q(s)ds\nonumber\\
&=&\int_{0}^{t}\sum_{i=0}^{k_n-1}v(t_i,Y)Q(s)\chi_{(t_i, t_{i+1}]}(s)dW(s) +\int_0^t\sum_{i=0}^{k_n-1}\D_s(v(t_i,Y))\chi_{(t_i, t_{i+1}]}(s)Q(s)ds\nonumber\\
&=&\int_{0}^{t}F^n(s)dW(s)+\int_0^tG^n(s)ds,
\end{eqnarray}
where
$$F^n(s):=\sum_{i=0}^{k_n-1}v(t_i,Y)Q(s)\chi_{(t_i, t_{i+1}]}(s), $$
$$G^n(s):=\sum_{i=0}^{k_n-1}\D_s(v(t_i,Y))\chi_{(t_i, t_{i+1}]}(s)Q(s), \quad 0 \leq s\leq t.$$
Recall that $\Bbb{L}^{1,2}(H)$ (see \cite{N})  is the class of $H$-valued processes $u$ such that $u(t)\in \D^{1,2}(H)$ for almost all $t$, and there exists a measurable version of the two parameter process $D_su(t)$ verifying $E[\int_0^T\int_0^T|D_su(t)|_H^2dsdt]<\infty$. We say that $u\in \Bbb{L}^{1,2}_{loc}(H)$ if there exists a sequence
$\{ (\Omega_n, u^n), n\geq 1\}\subset {\cal F}\times \Bbb{L}^{1,2}(H)$ such that $\Omega_n$ increases to $\Omega$ a.s. and $u=u^n$ a.e on $[0,T]\times \Omega_n$.
We first show that $F^n(\cdot)\rightarrow u(\cdot, Y)=v(\cdot, Y)Q(\cdot)$ in $\Bbb{L}^{1,2}_{loc}(H)$ as $n\rightarrow \infty$. To this end, we may assume without loss of generality that $|Y|_H\leq M$ for some constant $M$ and $Q=Q_N$. This is because, otherwise, we can replace $Y$ by
 $Y\phi(|Y|_N)$ where $\phi\in C_0^{\infty}(R)$ is a smooth bump function satisfying $\phi(x)=1$ whenever $|x|\leq M$ and $\phi(x)=0$ when $|x|> M+1$. Note that $v(s,Y)$ is continuous in $s$. It is clear that $F^n(s)\rightarrow u(s, Y)=v(s, Y)Q(s)$
for every $s\geq 0$. Moreover,
\begin{eqnarray}\label{e1.6}
\sup_{0\leq s\leq t}|F^n(s)|_H &\leq& (\sup_{0\leq s\leq t}|v(s, Y)|_H)(\sup_{0\leq s\leq t}Q(s))\nonumber\\
&\leq& |Y|_H(\sup_{0\leq s\leq t}Q(s))
\end{eqnarray}
The dominated convergence theorem yields that
\begin{equation}\label{e1.7}
\lim_{n\rightarrow \infty}E[\int_0^t |F^n(s)-u(s, Y)|_H^2ds]=0
\end{equation}
The Malliavin derivative of $F^n$ is given by
 \begin{eqnarray}\label{e1.8}
\D_uF^n(s)&=& \sum_{i=0}^{k_n-1}\D_u[v(t_i,Y)Q(s)]\chi_{(t_i, t_{i+1}]}(s) \nonumber\\
&=& \sum_{i=0}^{k_n-1}[\D_u(v(t_i,Y))Q(s)+v(t_i,Y)\D_uQ(s)]\chi_{(t_i, t_{i+1}]}(s)\nonumber\\
&=&\sum_{i=0}^{k_n-1}v(t_i,Y)\D_uQ(s)\chi_{(t_i, t_{i+1}]}(s)\nonumber\\
&+&\sum_{i=0}^{k_n-1}[\D_uv(t_i,Y)+Dv(t_i,Y)(\D_uY)]Q(s)\chi_{(t_i, t_{i+1}]}(s),
\end{eqnarray}
where $D v(s,f)$ stands for the Fr\'echet derivative of the mapping $v(s, \cdot)$ at the function $f$
and $\D_u v(t_i,Y)=\D_uv(t_i,f)\bigg |_{f=Y}$.
Since $v(s,Y)$, $\D_uv(s,Y)$ and $D v(s,Y)$ are continuous in $s$, it is easily seen that $\D_uF^n(s)\rightarrow \D_u(u(s, Y))=\D_u(v(s, Y)Q(s))$ for every $s\geq 0$. In view of Proposition 3.1, Proposition 3.2, and (3.30), it follows from (\ref{e1.8}) that
\begin{eqnarray}\label{e1.9}
|\D_uF^n(s)|_H&\leq & |Y|_H \sup_{0\leq s\leq t}|\D_uQ(s)| \chi_{[0, t]}(u)\nonumber\\
&& + c \sup_{0\leq s\leq t}|\D_uQ(s)| ||Q||_{\infty}|Y|_H^4 \exp(c||Q||_{\infty}^2|Y|_H^2)\chi_{[0, t]}(u)  \nonumber\\
&&+|\D_uY|_H \exp(c||Q||_{\infty}^2|Y|_H^2)||Q||_{\infty}
\end{eqnarray}
Thus from the dominated convergence theorem it follows that
\begin{equation}\label{e1.10}
\lim_{n\rightarrow \infty}E \bigg[\int_0^t \int_0^{\infty}|\D_uF^n(s)-\D_u(u(s, Y))|_H^2\,du\, ds \bigg ]=0
\end{equation}
The relations (\ref{e1.7}) and (\ref{e1.10}) imply that
$F^n(\cdot)\rightarrow u(\cdot, Y)=v(\cdot, Y)Q(\cdot)$ in $\Bbb{L}^{1,2}_{loc}(H)$ as $n\rightarrow \infty$. Consequently, we have
\begin{equation}\label{e1.11}
\lim_{n\rightarrow \infty}\int_0^tF^n(s)dW(s)=\int_0^tu(s,Y)dW(s),
\end{equation}
where the integrals in the above relation are Skorohod integrals. To compute
$\displaystyle \lim_{n \to \i} G^n(s)$, consider
\begin{eqnarray}\label{e1.12}
G^n(s)&=& Q(s)\sum_{i=0}^{k_n-1}[\D_sv(t_i,Y)+D v(t_i,Y)(\D_sY)]\chi_{(t_i, t_{i+1}]}(s) \nonumber\\
&=& Q(s)\sum_{i=0}^{k_n-1}D v(t_i,Y)(\D_sY)\chi_{(t_i, t_{i+1}]}(s),
\end{eqnarray}
where we have used the fact that $\D_sv(t_i,f)=0$ for $s>t_i$. Since $D v(s,Y)$ is continuous in $s$, we see that
\begin{equation}\label{e1.13}
\lim_{n\rightarrow \infty}\int_0^tG^n(s)ds=\int_0^tQ(s)D v(s,Y)(\D_sY)ds
\end{equation}
for all $t \geq 0$.

Putting (\ref{e1.2})--(\ref{e1.13}) together and letting $n\rightarrow \infty$, we arrive at
\begin{eqnarray}\label{e1.14}
u(t,Y)-Y&=&v(t,Y)Q(t)-Y\nonumber\\
&=&-\int_{0}^{t}Au(s,Y)ds-\int_{0}^{t}B(u(s,Y))ds+\int_{0}^{t}u(s,Y)dW(s)\nonumber\\
&&+ \frac{1}{2}\int_0^tu(s,Y)ds+\int_0^tQ(s)D v(s,Y)(\D_sY)ds
\end{eqnarray}
for all $t \geq 0$.

To complete the proof of the theorem, we need to show that
\begin{eqnarray}\label{e1.15}
\int_{0}^{t}u(s,Y)\circ dW(s)
&=&\int_{0}^{t}u(s,Y)dW(s)+ \frac{1}{2}\int_0^tu(s,Y)ds+\int_0^tQ(s){\cal D}v(s,Y)(\D_sY)\,ds \nonumber \\
\end{eqnarray}
for all $ t \geq 0$.
Define
$$\D^+_s u(s,Y):=\lim_{\varepsilon\rightarrow 0+}\D_s u(s+\varepsilon,Y),$$
$$D^-_s u(s,Y):=\lim_{\varepsilon\rightarrow 0+}\D_s u(s-\varepsilon,Y),$$
$$(\nabla u(\cdot,Y))(s):=\frac{1}{2}[\D^+_s u(s,Y) + D^-_s u(s,Y)]$$
for $s > 0$. Then, by Theorem 3.1.1 in \cite{N}, we know that
\begin{eqnarray}\label{e1.16}
\int_{0}^{t}u(s,Y)\circ dW(s)
&=&\int_{0}^{t}u(s,Y)dW(s)+ \frac{1}{2}\int_0^t(\nabla u(\cdot,Y))_sds, \quad t \geq 0.
\end{eqnarray}
Thus, it remains to show that
\begin{equation}\label{e1.17}
\frac{1}{2}(\nabla u(\cdot,Y))(s)=\frac{1}{2}u(s,Y)+Q(s)Dv(s,Y)(\D_sY)
\end{equation}
The Malliavin derivative of $u(t,Y)$ is given by
\begin{eqnarray}\label{e1.18}
\D_s(u(t,Y))&=&\D_s(v(t,Y)Q(t))\nonumber\\
&=&\D_s(v(t,Y))Q(t)+v(t,Y)\D_sQ(t)\nonumber\\
&=&\D_sv(t,Y)Q(t)+Dv(t,Y)(\D_sY)Q(t)+v(t,Y)\D_sQ(t)
\end{eqnarray}
for all $t \geq 0$.  Replacing $t$ by $s-\varepsilon$ in (\ref{e1.18}) we get
$$\D_s(u(s-\varepsilon,Y))={\cal D}v(s-\varepsilon,Y)(\D_sY)Q(s-\varepsilon ),$$
where we have used the fact that $\D_sv(s-\varepsilon,Y)=0$, $\D_sQ(s-\varepsilon )=0$. This yields that
\begin{equation}\label{e1.19}
D^-_s(u(\cdot, Y))(s)=Dv(s,Y)(\D_sY)Q(s).
\end{equation}
Next, we replace $t$ by $s+\varepsilon$ in (\ref{e1.18}), let $\varepsilon \rightarrow 0$ and use the continuity of the
the functions involved to obtain
\begin{equation}\label{e1.20}
D^+_s(u(\cdot, Y))(s)=Dv(s,Y)(\D_sY)Q(s)+u(s,Y),
\end{equation}
where we have used the facts $\lim_{\varepsilon\rightarrow 0}\D_sv(s+\varepsilon, Y)=0$ and $\D_sQ(s+\varepsilon)=Q(s+\varepsilon)$. Finally, (\ref{e1.17}) follows from (\ref{e1.19}) and (\ref{e1.20}).

\end{proof}


\bigskip
\begin{thebibliography}{[M-N-S]}
\addcontentsline{toc}{chapter}{References}

\bibitem[B-C-F]{B-C-F}
Z. Brze\'zniak, M. Capinski, and F. Flandoli,
\textit{Stochastic Navier-Stokes equations with multiplicative noise}, Stochastic Analysis and Applications \textbf{10} (1992), 53--532.

\bibitem[B-L]{B-L} Z. Brze\'zniak, and Y. Li, \textit{Asymptotic compactness and absorbing sets for 2D stochastic Navier-Stokes equations on some unbounded domains}  Trans. Amer. Math. Soc.  358  (2006),  no. 12, 5587--5629

\bibitem[D-Z.1]{D-Z.1}
G. Da Prato and J. Zabczyk,
\textit{Stochastic Equations in Infinite Dimensions},
Cambridge University Press, Cambridge, 1992.

\bibitem[D-Z.2]{D-Z.2}
G. Da Prato and J. Zabczyk,
\textit{Ergodicity for Infinite Dimensional Systems},
Cambridge University Press, Cambridge, 1996.

\bibitem[D-L-S.1]{D-L-S.1}
J. Duan, K. Lu, and B. Schmalfuss,
\textit{Invariant manifolds for stochastic partial differential equations},
Annals of Probability \textbf{31} (2003), 2109--2135.

\bibitem[Fl.1]{Fl.1}F. Flandoli, \textit{Stochastic differential equations in fluid dynamics},  Rend. Sem. Mat. Fis. Milano  66  (1996), 121--148.

\bibitem[Gourcy]{Gourcy} M. Gourcy, \textit{A large deviation principle for 2D stochastic Navier-Stokes equation},  Stochastic Process. Appl.  117  (2007),  no. 7, 904--927.

\bibitem[H-M]{H-M} Hairer, M., and Mattingly, J. C., \textit{Ergodicity of the 2D Navier-Stokes equations with degenerate stochastic forcing}, Ann. of Math. (2)  164
(2006),  no. 3, 993--1032.

\bibitem[M-Z]{M-Z} Mohammed, S.-E. A., Zhang,
T. S., \textit{Dynamics of stochastic 2D Navier-Stokes equations }, Journal of Functional Analysis,  (2008), pp. 105.



\bibitem[N]{N} Nualart, D.,
T. S., \textit{The Malliavin Calculus and Related Topics }, Springer-Verlag, 1995.


\bibitem[Ro]{Ro}
J. C.  Robinson,
\textit{Infinite-Dimensional Dynamical Systems}, Cambridge Texts in Applied Mathematics,
Cambridge University Press, Cambridge, 2001.

\bibitem[S-S]{S-S} Sritharan, S. S. and Sundar, P. \textit{Large deviations for the two-dimensional Navier-Stokes equations with multiplicative noise}, Stochastic Process. Appl. 116 (2006), no. 11, 1636--1659.


\bibitem[Te]{Te}
R. Temam,
\textit{Infinite-Dimensional Dynamical Systems in Mechanics and Physics},
Springer-Verlag, New York, 1988.

\bibitem[Te1]{Te1}
R. Temam, \textit{Navier-Stokes Equations}, North-Holland Amsterdam, New York, Oxford,
1979.
\end{thebibliography}
\end{document}